\documentclass[a4,reqno,12pt]{amsart}
\usepackage{amsmath,amsthm,amssymb,color}
\usepackage{multicol}

\newcommand{\R}{\mathbb{R}}
\newcommand{\C}{\mathbb{C}}
\newcommand{\U}{\mathcal{U}}
\newcommand{\ue}{\beta}                                  
\newcommand{\naka}{\beta}
\newcommand{\sita}{\beta_0}
\newcommand{\bmin}{\beta_{{\rm min}}}
\newcommand{\bint}{\beta_{{\rm int}}}
\newcommand{\Fbel}[1]{\widetilde{\mu_{\varphi,{#1}}}}
\newcommand{\Fru}[1]{F_{\rho,{#1}}}

\newcommand{\ru}{\rho} 
\newcommand{\n}[1]{\left\vert {#1} \right\vert}                           %
\newcommand{\N}[1]{\left\Vert {#1} \right\Vert}                           
\newcommand{\inner}[2]{\left\langle {#1} , {#2} \right\rangle}            
\newcommand{\Vol}[1]{{\rm Vol} \left( {#1} \right)}                   
\newcommand{\Volb}[1]{{\rm Vol}_{\beta} \left( {#1} \right)}               
\newcommand{\UDL}[1]{\underline{ {#1} }}                   %
\newcommand{\Siegel}{\mathcal{D}}
\newcommand{\cent}{t}
\newcommand{\Kr}{C}

\date{}
\newtheorem{Def}{Definition}[section]
\newtheorem{Thm}[Def]{Theorem}
\newtheorem{Lem}[Def]{Lemma}
\newtheorem{IntroThm}{Theorem}

\newtheorem{Cor}[Def]{Corollary}
\newtheorem{?}[Def]{Question}

\makeatletter                  
 
   \@addtoreset{equation}{section}
\makeatother

\title[Composition operators of a minimal domain]
{Composition operators on the Bergman spaces of a minimal bounded homogeneous domain} 
\author[S. Yamaji]{Satoshi Yamaji}

\address{%
Satoshi Yamaji \endgraf
Graduate School of Mathematics \endgraf
Nagoya University \endgraf
Chikusa-ku, Nagoya, 464-8602 \endgraf
Japan
}

\email{satoshi.yamaji@math.nagoya-u.ac.jp}


\begin{document}

\begin{abstract}
Using an integral formula on a homogeneous Siegel domain,
 we show a necessary and sufficient condition for composition operators on the weighted Bergman space
 of a minimal bounded homogeneous domain to be compact. 
To describe the compactness of composition operators, we see a boundary behavior of the Bergman kernel.
\end{abstract}

\maketitle

\section{Introduction}
In 2007, Zhu \cite{ZhuComp} considered the composition operators on the weighted Bergman space of the unit ball. 
His results are extended to the case that the domain is the Harish-Chandra realization of 
irreducible bounded symmetric domain by Lu and Hu \cite{LuHu}. 
In this paper, we consider a generalization of their works for the weighted Bergman space of a minimal bounded homogeneous domain 
(for the definition of the mininal domain, see \cite{I-Y}, \cite{MM}). 
Indeed, the unit ball, the polydisk and a bounded symmetric domain in its Harish-Chandra realization 
are minimal domains.

Let $\mathcal{U}$ be a minimal bounded homogeneous domain in $\C^d$, $dV(z)$ the Lebesgue measure on 
$\C^d$ and $\mathcal{O}(\mathcal{U})$ the space of all holomorphic functions on $\mathcal{U}$. 
The Bergman kernel $K_{\U} : \U \times \U \longrightarrow \C$ is the reproducing kernel of the Bergman space 
$L^2_a(\U,dV) := L^2(\U,dV) \cap \mathcal{O}(\U)$. 
The Bergman kernel is a useful tool to study properties of composition operators, Toeplitz operators and Hankel operators 
on the Bergman space (for example, see \cite{Zhu1}). 
In this paper, we see that  a necessary and sufficient condition for a bounded composition operator 
to be compact is described by a boundary behavior of the Bergman kernel.

For $\beta \in \R$, let $dV_{\beta}$ denote the measure on $\U$ given by 
$dV_{\beta}(z) := K_{\mathcal{U}}(z,z)^{-\beta} dV(z)$. 
We consider the weighted Bergman space 
$L^p_a(\U,dV_{\beta}) := L^p(\mathcal{U},dV_{\beta}) \cap \mathcal{O}(\mathcal{U})$. 
It is known that there exists a constant $\bmin$ such that $L^p_a(\U,dV_{\beta})$ is non-trivial 
if and only if $\beta > \bmin$ (for explicit expression of $\bmin$, see section \ref{sect22}). 
From now on, we consider non-trivial weighted Bergman spaces.
For a holomorphic map $\varphi$ from $\mathcal{U}$ to $\mathcal{U}$, the composition operator $C_{\varphi}$ is 
a linear operator on $\mathcal{O} (\mathcal{U})$ defined by $C_{\varphi} f := f \circ \varphi$. 
We conside the composition operator on the weighted Bergman space $L^p_a(\U,dV_{\beta})$. 
Using Zhu's technique (see \cite{ZhuComp}) together with an integral formula (see Lemma \ref{CulSiegel}), 
we obtain the following theorem, which is the main theorem of this paper. 
\begin{IntroThm}[Theorem \ref{MainComp}]\ \label{MainComp0}
Assume that 
$C_{\varphi}$ is bounded on $L^q_a(\mathcal{U},dV_{\sita})$ for some $q>0$ and $\sita > \bmin$. 
Then $C_{\varphi}$ is compact on $L^p_a(\mathcal{U},dV_{\naka})$ for any $p>0$ and $\naka > \sita + \bint$ if and only if 
\begin{align*}
 \lim_{ z \rightarrow \partial \mathcal{U}} \frac{K_{\mathcal{U}} \left( \varphi(z), \varphi(z) \right) }{K_{\mathcal{U}}(z,z)} = 0 . 
\end{align*}
\end{IntroThm}

Since the unit ball and the Harish-Chandra realization of irreducible bounded symmetric domain are 
minimal domains, 
Theorem \ref{MainComp0} is a generalization of \cite[Theorem 4.1]{ZhuComp} and \cite[Theorem]{LuHu} 
(see section \ref{sect60}). 
Similarly to the case of them, 
the assumption that $C_{\varphi}$ is a bounded operator on $L^p_a(\mathcal{U},dV_{\sita})$ for some 
$\sita > \bmin$ 
is needed only for the ``if'' part of Theorem \ref{MainComp0}. 

To prove Theorem \ref{MainComp0} for the case that $\U = \mathbb{B}^d$, Zhu used Schur's theorem. 
To apply Zhu's method, it is important to find a positive function satisfying a certain inequality. 
Zhu found this function by using Forelli-Rudin inequality (see \cite[Lemma 2.6]{ZhuComp}). 
Instead, we find the function by using Lemma \ref{CulSiegel}. 
By \cite{VGS}, there exists a biholomorphic map $\Phi$ from the bounded homogeneous domain $\U$ onto 
a homogeneous Siegel domain $\mathcal{D}$. In Lemma \ref{CulSiegel}, we shall consider the integral 
\begin{align*}
\int_{\mathcal{U}} \n{K_{\mathcal{U}} (z,z^{\prime})}^{1+\alpha}   \n{\det J(\Phi,z^{\prime})}^{1+2\beta-\alpha}  \, dV_{\beta}(z^{\prime}) ,
\end{align*}
where $J(\Phi,z^{\prime})$ denotes the complex Jacobi matrix of $\Phi$ at $z^{\prime}$. 
The integral converges if and only if $\naka > \bmin$ and $\alpha > \naka + \bint$, 
where $\bint$ is a constant defined from $\U$ (see section \ref{sect22}).

Before the proof of theorem \ref{MainComp0}, we show that the boundedness of $C_{\varphi}$ 
on $L^p_a(\mathcal{U},dV_{\beta})$ is described 
in terms of Carleson measures. It is easy to see that $C_{\varphi}$ is a bounded operator on $L^p_a(\U,dV_{\beta})$ if and only if the pull-back measure $d\mu_{\varphi,\beta}$ of $dV_{\beta}$ 
induced by $\varphi$ is a Carleson measure for $L^p_a(\U,dV_{\beta})$ (see section \ref{sect41}).
Using properties of Carleson measures, we obtain the following theorem. 
\begin{IntroThm}[Theorems \ref{BdCp} and \ref{CptCp}]\ \label{BdAndCptCp}
If $C_{\varphi}$ is a bounded (resp. compact) operator on $L^q_{a}(\U,dV_{\sita})$ for some $q>0$ and 
$\sita> \bmin$, then $C_{\varphi}$ is a bounded (resp. compact) operator on $L^p_{a}(\U,dV_{\naka})$ 
for any $p>0$ and $\naka \geq \sita$. 
\end{IntroThm}

By Theorem \ref{BdAndCptCp}, the assumption of Theorem \ref{MainComp0} implies that 
$C_{\varphi}$ is bounded on $L^q_a(\mathcal{U},dV_{\sita})$ for any $q>0$ and $\beta \geq \sita$. 
We use the boundedness of $C_{\varphi}$ on $L^2_a(\mathcal{U},dV_{\beta})$ and $L^2_a(\mathcal{U},dV_{\sita})$ 
in section \ref{sect-Lemma}. 

Let us explain the organization of this paper. 
In section \ref{sect20}, we review properties of the weighted Bergman space of a minimal bounded homogeneous 
domain and composition operators on the space. 
Theorem \ref{Ishi} plays an important role in this section. 
In section \ref{sect30}, we show some properties of Carleson measures and vanishing Carleson measures for 
the weighted Bergman space of a minimal bounded homogeneous domain (Theorems \ref{W-Carlprop} and \ref{W-VanCarl}). 
Using them, we prove properties of the boundedness and compactness of $C_{\varphi}$ in section \ref{sect40} (Theorems \ref{BdCp} and \ref{CptCp}). 
In section \ref{sect4.00}, we show an important equality (Lemma \ref{CulSiegel}). 
By using Lemma \ref{CulSiegel}, we prove the characterization of the compactness of 
the composition operator (Theorem \ref{MainComp0}) in section \ref{sect50}. 
In section \ref{sect60}, we apply Theorem \ref{MainComp0} for the case that 
$\mathcal{U}$ is the unit ball, bounded symmetric domain in its Harish-Chandra realization, the polydisk and 
the representative domain of the tube domain over Vinberg cone, which is an example of nonsymmetric bounded homogeneous domain. 
These domains are minimal domains with a center $0$.

\section{Preliminaries} \label{sect20}
\subsection{Weighted Bergman spaces of a minimal bounded homogeneous domain} \label{sect21}
Let $D$ be a bounded domain in $\C^d$. 
We say that $D$ is a minimal domain with a center $t \in D$ if the following condition is satisfied: 
for every biholomorphism $\psi : D \longrightarrow D^{\prime}$ with $\det J(\psi , t)=1$, we have 
\begin{eqnarray*}
{\rm Vol} (D^{\prime}) \geq {\rm Vol} (D)  .
\end{eqnarray*}
We see that $D$ is a minimal domain with a center $t$ if and only if 
$$K_{D}(z,t) = \frac{1}{{\rm Vol }(D)}$$
for any $z \in D$ (see \cite[Proposition 3.6]{I-K} or \cite[Theorem 3.1]{MM}). 
For example, the unit disk $\mathbb{D}$ and the unit ball $\mathbb{B}^d$ are minimal domains with a center $0$. 

We fix a minimal bounded homogeneous domain $\mathcal{U}$ with a center $\cent$. 
We denote by $K_{\mathcal{U}}^{(\beta)}$ the reproducing kernel of $L^2_a(\U,dV_{\beta})$. 
It is known that $K_{\mathcal{U}}^{(\beta)}(z,w) = C_{\beta} K_{\mathcal{U}}(z,w)^{1+\beta}$ 
for some positive constant $C_{\beta}$. 
For $z \in \mathcal{U}$, we denote by $k_z^{(\beta)}$ the normalized reproducing kernel of $L^2_a(\U,dV_{\beta})$, that is, 
\begin{align}
   k_z^{(\beta)} (w) := \frac{K_{\mathcal{U}}^{(\beta)}(w,z)}{K_{\mathcal{U}}^{(\beta)}(z,z)^{\frac{1}{2}}} 
    = \sqrt{C_{\beta}} \left( \frac{K_{\mathcal{U}}(w,z)}{K_{\mathcal{U}}(z,z)^{\frac{1}{2}}} \right)^{1+\beta} . \label{WBergKer}
\end{align}
For any Borel set $E$ in $\U$, we define
$$  \Volb{E} := \int_{E} dV_{\beta}(w) .$$
Let $d_{\U}(\cdot,\cdot)$ be the Bergman distance on $\mathcal{U}$. 
For any $z \in \mathcal{U}$ and $r>0$, let
$$  B(z,r) := \{ w \in \mathcal{U} \mid d_{\U}(z,w) \leq r \} $$
be the Bergman metric disk with center $z$ and radius $r$. 

%
In \cite{I-Y}, we proved the following theorem. 
\begin{Thm}[\mbox{\cite[Theorem A]{I-Y}}] \label{Ishi} 
For any $\ru>0$, there exists $C_{\ru}>0$ such that 
$$  C_{\ru}^{-1} \leq \n{\frac{K_{\mathcal{U}}(z,a)}{K_{\mathcal{U}}(a,a)} }\leq C_{\ru} $$
for all $z,a \in \mathcal{U}$ such that $d_{\U} (z,a) \leq \ru$. 
\end{Thm}

From Theorem \ref{Ishi}, we obtain that $K_{\mathcal{U}}(\cdot,a)$ is a bounded function on $\mathcal{U}$ 
for each $a \in \mathcal{U}$ (see \cite[Proposition 6.1]{I-Y}). 
Since ${\rm span}\langle K^{(\beta)}_{\mathcal{U}}(\cdot,a) \rangle$ is dense in $L^2_a(\mathcal{U}, dV_{\beta})$, 
we see that $H^{\infty}(\mathcal{U})$ is dense in $L^2_a(\mathcal{U}, dV_{\beta})$, 
where $H^{\infty}(\mathcal{U})$ is the set of all bounded holomorphic functions on $\mathcal{U}$. 

Moreover, we obtain useful lemmas from Theorem \ref{Ishi}. 
First, we deduce 
\begin{align}
  C_{\ru}^{-2} \leq \frac{K_{\mathcal{U}}(z,z)}{K_{\mathcal{U}}(a,a)} \leq C_{\ru}^2 \label{ineq2}
 \end{align}
for all $z,a \in \mathcal{U}$ such that $d_{\U} (z,a) \leq \ru$. 
On the other hand, we have  
\begin{align*}
C^{-1} K_{\U}(a,a)^{-1} \leq \n{\frac{K_{\U}(z,a)}{K_{\U}(a,a)}}^2 \Vol{B(a,\ru)} \leq C K_{\U}(a,a)^{-1} . 
\end{align*}
by \cite[Lemma 3.3]{Yamaji}. Therefore, we have 
\begin{align}
C^{-1} K_{\U}(a,a)^{-1} \leq \Vol{B(a,\ru)} \leq C K_{\U}(a,a)^{-1}  \label{Nweight2}
\end{align}
by Theorem \ref{Ishi}. 
Therefore, we have the following lemma.

\begin{Lem} \label{WVOL}
There exists a positive constant $C$ such that 
\begin{align}
  C^{-1} K_{\U}(a,a)^{-(1+\beta)} \leq  \Volb{B(a,\ru)} \leq C  K_{\U}(a,a)^{-(1+\beta)}  \label{Nweight}
\end{align}
for all $a \in \mathcal{U}$.
\end{Lem}
\begin{proof}
Since
\begin{align*}
\Volb{B(a,\ru)} 
  &= \int_{B(a,\ru)} K_{\U}(w,w)^{-\beta} dV(w) ,  
\end{align*}
we have 
\begin{align}
C^{-1} K_{\U}(a,a)^{-\beta}  \Vol{B(a,\ru)} \leq \Volb{B(a,\ru)} 
\leq C K_{\U}(a,a)^{-\beta}  \Vol{B(a,\ru)}. \label{Nweight3}
\end{align}
by (\ref{ineq2}). We obtain (\ref{Nweight}) from (\ref{Nweight2}) and (\ref{Nweight3}).
\end{proof}

By (\ref{WBergKer}), Lemma \ref{WVOL} and Theorem \ref{Ishi}, we have the following lemma. 
\begin{Lem}[cf. \mbox{\cite[Lemma 1]{Zhu1}}] \label{W-Lemma1} 
There exists a positive constant $\Kr$ such that 
$$  \Kr^{-1} \leq \n{k_a^{(\beta)} (z)}^2 \Volb{B(a,\ru)} \leq \Kr $$
for all $a \in \mathcal{U}$ and $z \in B(a,\ru)$.
\end{Lem}

Lemma \ref{WVOL}  and (\ref{Nweight2}) yield the following;
\begin{Lem}[cf. \mbox{\cite[Lemma 2]{Zhu1}}] \label{W-Vol} 
There exists a positive constant $C$ such that 
$$  C^{-1} \Volb{B(a,\ru)} \leq \Volb{B(z,\ru)} \leq C \Volb{B(a,\ru)} $$
for all $a \in \mathcal{U}$ and $z \in B(a,\ru)$.
\end{Lem}

We have the following estimate. 
\begin{Lem}[cf. \mbox{\cite[Lemma 5]{Zhu1}}]  \label{W-Lemma7}
There exists a positive constant $C$ such that 
\begin{align}
 \n{f(z)}^p 
  \leq \frac{C}{\Volb{B(z,\ru)}}  \int_{B(z,\ru)} \n{f(w)}^p \, dV_{\beta}(w) \label{WWW}
\end{align}
for all $f \in \mathcal{O} (\mathcal{U}), p >0$ and $z \in \mathcal{U}$. 
\end{Lem}
\begin{proof}
By \cite[Lemma 3.5]{Yamaji}, there exists a $C>0$ such that 
\begin{align*}
\n{f(z)}^p 
  & \leq \frac{C}{\Vol{B(z,\ru)}} \int_{B(z,\ru)} \n{f(w)}^p \, dV(w) \\
  & \leq \frac{C K_{\U}(z,z)^{\beta}}{\Vol{B(z,\ru)}} \int_{B(z,\ru)} \n{f(w)}^p 
  \, dV_{\beta}(w) ,
\end{align*}
where the lat inequality follows from (\ref{ineq2}). 
By (\ref{Nweight3}), we have 
\begin{align*}
\frac{K_{\U}(z,z)^{\beta}}{\Vol{B(z,\ru)}} 
  \leq \frac{C}{\Volb{B(z,\ru)}} .
\end{align*}
Hence, we obtain (\ref{WWW}). 
\end{proof}

\subsection{Composition operator} \label{sect24}
In this section, we summarize properties of the composition operator (see also \cite[section 11]{Zhu3}, \cite{ZhuComp}). 
Let $\varphi$ be a holomorphic map from $\mathcal{U}$ to $\mathcal{U}$. 
For $f \in \mathcal{O} ( \mathcal{U})$, we define $C_{\varphi} f := f \circ \varphi$. 
Then, $C_{\varphi}$ is a linear operator on $\mathcal{O} (\mathcal{U})$. 
The operator $C_{\varphi}$ is called the composition operator induced by $\varphi$. 
It is known that $C_{\varphi}$ is always bounded on $L^p_a(\mathcal{U},dV_{\beta})$ 
for the case that $\mathcal{U}$ is the unit disk $\mathbb{D}$. 
However, for a general minimal bounded homogeneous domain $\U$, a composition operator is not 
necessarily bounded on $L^p_a(\mathcal{U},dV_{\beta})$ (for example, see \cite{ZhuComp}).

On the other hand, for any Borel set $E$ in $\mathcal{U}$, we define 
$$ \mu_{\varphi,\beta} (E) := {\rm Vol}_{\beta}\,(\varphi^{-1}(E)) . $$ 
The measure $\mu_{\varphi,\beta}$ is called the pull-back measure of $dV_{\beta}$ induced by $\varphi$. 
Then, $C_{\varphi}$ is a bounded operator on $L^p_a(\mathcal{U},dV_{\beta})$ if and only if there exists 
a constant $C>0$ such that 
\begin{align}
 \int_{\U} \n{f(w)}^p \, d\mu_{\varphi,\beta}(w) 
  \leq C  \int_{\U} \n{f(w)}^p \, dV_{\beta}(w) \label{DefC2}
\end{align}
holds for any $f \in L^p_a(\mathcal{U},dV_{\beta})$. 

Assume that $C_{\varphi}$ is a bounded operator on $L^2_a(\mathcal{U},dV_{\beta})$. Then, we have 
\begin{align}
C_{\varphi}^{\ast} f(w) 
  &= \langle C_{\varphi}^{\ast} f , K_{w}^{(\beta)} \rangle_{L^2(dV_{\beta})} 
  = \inner{f}{C_{\varphi}K_{w}^{(\beta)}}_{L^2(dV_{\beta})}  \label{Adj}
\end{align}
for any $f \in L^2_a(\mathcal{U},dV_{\beta})$. Therefore, we have
\begin{align}
C_{\varphi} C_{\varphi}^{\ast} f(w) 
  = \langle  f, C_{\varphi}K_{\varphi(w)}^{(\beta)} \rangle_{L^2(dV_{\beta})}  
  =   \int_{\U}  K_{\U}^{(\beta)} (\varphi (w), \varphi (u)) f(u) \, dV_{\beta}(u)  . \label{CptKakikae}
\end{align}
We use (\ref{CptKakikae}) to characterize the compactness of $C_{\varphi}$. 
Moreover, we have 
\begin{align*}
C_{\varphi}^{\ast} C_{\varphi} f(w) 
  &= \inner{C_{\varphi}f}{C_{\varphi}K_{w}^{(\beta)}}_{L^2(dV_{\beta})}  \\
  &= \int_{\U} f(\varphi (u)) K_{\U}^{(\beta)} (w, \varphi (u)) \, dV_{\beta}(u)  \\
  &= \int_{\U} K_{\U}^{(\beta)} (w, u) f(u) \, d\mu_{\varphi,\beta}(u)  
\end{align*}
by (\ref{Adj}). Therefore, we obtain $C_{\varphi}^{\ast} C_{\varphi} = T_{\mu_{\varphi,\beta}}$, 
where $T_{\mu_{\varphi,\beta}}$ is the Toeplitz operator with symbol $\mu_{\varphi,\beta}$. 
The boundedness of Toeplitz operators are discussed in \cite{Yamaji}, \cite[section 7]{Zhu3} and \cite{ZhuComp}.


\section{Carleson measures and vanishing Carleson measures}  \label{sect30}
\subsection{Berezin symbol and averaging function}
For a Borel measure $\mu$ on $\mathcal{U}$, we define a function $\widetilde{\mu}$ on $\mathcal{U}$ by
\begin{eqnarray*}
 \widetilde{\mu}(z) := \int_{\mathcal{U}} \vert k_z^{(\beta)}(w) \vert^2 \, d\mu(w) ,
\end{eqnarray*}
which is called the Berezin symbol of the measure $\mu$. 
For fixed $\ru>0$, we define a function $\widehat{\mu}$ on $\mathcal{U}$ by
\begin{eqnarray*}
\widehat{\mu}(z) := \frac{\mu(B(z,\ru))}{\Volb{B(z,\ru)}} ,
\end{eqnarray*}
which is called the averaging function of the Borel measure $\mu$. 
Although the value of $\widehat{\mu}$ depends on the parameter $\ru$, we will ignore that distinction. 
\begin{Lem} \label{CarlesonLemma}
There exists a positive constant $C$ such that 
\begin{align*}
 \int_{\U} \n{f(z)}^p \, d \mu(z) 
  \leq C \int_{\U} \widehat{\mu}(z) \n{f(z)}^p \, dV_{\beta}(z) 
\end{align*}
for any $p>0$ and $f \in \mathcal{O} (\mathcal{U})$. 
\end{Lem}
\begin{proof}
By Lemma \ref{W-Lemma7}, we have 
\begin{align}
\int_{\U} \n{f(z)}^p \, d \mu(z) 
  \leq C \int_{\U} \left( \frac{1}{\Volb{B(z,\ru)}}  \int_{B(z,\ru)} \n{f(w)}^p \, dV_{\beta}(w) \right) d \mu(z)  \label{Carl3}
\end{align}
for any $p>0$ and $f \in \mathcal{O} (\mathcal{U})$. 
The right hand side of (\ref{Carl3}) is equal to
\begin{align}
 C \int_{\U}\int_{\U} \frac{\chi_{B(z,\ru)}(w)}{\Volb{B(z,\ru)}} \n{f(w)}^p \, dV_{\beta}(w)d \mu(z) .  \label{Carl4}
\end{align}
By using Fubini's theorem, (\ref{Carl4}) is equal to
\begin{align}
 C \int_{\U} \left( \int_{{B(w,\ru)}} \frac{1}{\Volb{B(z,\ru)}}  \, d \mu(z) \right) \n{f(w)}^p \, dV_{\beta}(w) . \label{Carl5}
\end{align}
By Lemma \ref{W-Vol}, (\ref{Carl5}) is less than or equal to 
\begin{align*}
 C \int_{\U}  \frac{\mu(B(w,\ru))}{\Volb{B(w,\ru)}}  \n{f(w)}^p \, dV_{\beta}(w)   .  
\end{align*}
\end{proof}

\subsection{Carleson measures}
Let $\mu$ be a positive Borel measure on $\mathcal{U}$ and $p >0$. 
We say that $\mu$ is a Carleson measure for $L^p_a(\mathcal{U},dV_{\beta})$ if 
there exists a constant $M>0$ such that 
\begin{eqnarray*}
 \int_{\mathcal{U}} \vert f(z) \vert^p \, d\mu (z) \leq M \int_{\mathcal{U}} \vert f(z) \vert^p \, dV_{\beta} (z) 
\end{eqnarray*}
for all $f \in L^p_a(\mathcal{U},dV_{\beta})$. 
It is easy to see that $\mu$ is a Carleson measure for $L^p_a(\mathcal{U},dV_{\beta})$ 
if and only if $L^p_a(\mathcal{U},dV_{\beta}) \subset L^p_a(\mathcal{U}, d\mu)$ and the inclusion map 
$$  i_p : L^p_a(\mathcal{U},dV_{\beta}) \longrightarrow  L^p_a(\mathcal{U}, d\mu)$$
is bounded. 

The following theorem is a generalization of \cite[Theorem 7]{Zhu1} to a minimal bounded homogeneous domain. 
\begin{Thm} \label{W-Carlprop} 
Let $\mu$ be a positive Borel measure on $\mathcal{U}$. 
Then, the following conditions are all equivalent. \\
  (i) \ $\mu$ is a Carleson measure for $L^p_a(\mathcal{U},dV_{\beta})$. \\
 (ii) \ $\widetilde{\mu}$ is a bounded function on $\U$. \\
(iii) \ $\widehat{\mu}$ is a bounded function on $\U$. 
\end{Thm}

\begin{proof}
First, we prove $(i) \Longrightarrow (ii)$. 
Since $k^{(\beta)}_z(w)^{\frac{2}{p}} \in L^p_a(\mathcal{U},dV_{\beta})$ and $\mu$ is a Carleson measure for $L^p_a(\mathcal{U},dV_{\beta})$, 
we have 
\begin{align*}
\int_{\mathcal{U}} \vert k_z^{(\beta)}(w) \vert^2 \, d\mu(w) 
\leq M \int_{\mathcal{U}}\n{k^{(\beta)}_z(w)}^2  dV_{\beta} (w) =M .
\end{align*}
Therefore, $\widetilde{\mu}$ is bounded. 
Next, we prove $(ii) \Longrightarrow (iii)$. 
Take any $w \in \mathcal{U}$. By Lemma \ref{W-Lemma1}, there exists a positive constant $C$ such that 
\begin{eqnarray}
  C \leq \n{k^{(\beta)}_z(w)}^2 \Volb{B(w,\ru)}   \label{W-Carlesonequiv2}
\end{eqnarray}
holds for any $w \in B(z,\ru)$. 
We integrate $(\ref{W-Carlesonequiv2})$ on $B(z,\ru)$ by $d\mu$. Then, we have 
\begin{align}
 \frac{\mu(B(z,\ru))}{\Volb{B(z,\ru)}}   
   \leq C \int_{B(a,\ru)}\n{k^{(\beta)}_a(z)}^2 \, d\mu(z) .   \label{W-Carlesonequiv2.3}
\end{align}
Therefore, we have 
\begin{align}
\widehat{\mu}(z) \leq C \widetilde{\mu}(z) .  \label{W-Carlesonequiv3}
\end{align}
Hence,  $(ii) \Longrightarrow (iii)$ holds. The part $(iii) \Longrightarrow (i)$ follows from Lemma \ref{CarlesonLemma}. 
\end{proof}
Similarly to \cite[Theorem 4.1]{Yamaji}, we can prove that these conditions are equivalent to the following condition: 
(iv) The Toeplitz operator $T_{\mu}$ is bounded on $L^2_a(\mathcal{U},dV_{\beta})$.

\subsection{Vanishing Carleson measure}
Suppose that $\mu$ is a Carleson measure for $L^p_a(\mathcal{U},dV_{\beta})$. 
We say that $\mu$ is a vanishing Carleson measure for $L^p_a(\mathcal{U},dV_{\beta})$ if 
\begin{align*}
\lim_{k \rightarrow \infty} \int_{\U} \n{f_k(w)}^p \, d \mu(w) =0 
\end{align*}
whenever $\{ f_k \}$ is a bounded sequence in $L^p_a(\mathcal{U}, dV_{\beta})$ that converges to $0$ 
uniformly on each compact subset of $\U$. 

The following theorem is a generalization of \cite[Theorem 11]{Zhu1} to a minimal bounded homogeneous domain. 
\begin{Thm}  \label{W-VanCarl} 
Let $\mu$ be a finite positive Borel measure on $\mathcal{U}$. 
Then, the following conditions are all equivalent. \\
  (i) \ $\mu$ is a vanishing Carleson measure for $L^p_a(\mathcal{U},dV_{\beta})$. \\
 (ii) \ $\widetilde{\mu} (z) \rightarrow 0$ as $z \rightarrow \partial \U$. \\
(iii) \ $\widehat{\mu} (z) \rightarrow 0$ as $z \rightarrow \partial \U$.
\end{Thm}

\begin{proof}
First, we prove $(i) \Longrightarrow (ii)$. 
In the same way as in \cite[Lemma 1]{Eng} and \cite[Lemma 5]{Eng}, 
we can see that $\{ k^{(\beta)}_z \}$ converges to $0$ uniformly on compact subsets of $\U$ as 
$z \rightarrow \partial \U$. 
Therefore, $\{ k^{(\beta)}_z(w)^{\frac{2}{p}} \}$ is a bounded sequence in $L^p_a(\mathcal{U}, dV_{\beta})$ 
that converges to $0$ uniformly on each compact subset of $\U$. 
Hence, $(ii)$ holds. The part $(ii) \Longrightarrow (iii)$ follows from (\ref{W-Carlesonequiv3}).
Finally, we prove $(iii) \Longrightarrow (i)$. 
Take any bounded sequence $\{ f_n \}$ in $L^p_a(\mathcal{U}, dV_{\beta})$ that converges to $0$ 
uniformly on each compact subset of $\U$. 
Take any $\varepsilon >0$. Then, there exists a constant $\delta>0$ such that 
\begin{align*}
\sup_{{\rm dist}(z,\partial \U)< \delta} \n{\widehat{\mu}(z)} < \varepsilon 
\end{align*}
by (iii). Let 
$\U_{\delta} := \{ z \in \mathcal{U} \mid {\rm dist}(z,\partial \U) < \delta \}$. 
Since $\U \backslash \U_{\delta}$ is a compact set, there exists an integer $N$ such that 
\begin{align*}
\sup_{z \in \U \backslash \U_{\delta}} \n{f_n(z)}^p < \varepsilon 
\end{align*}
for any $n \geq N$. Here, we have 
\begin{align}
\int_{\U} \n{f_n(z)}^p \, d\mu(z) 
   & \leq C \int_{\U} \widehat{\mu}(z)  \n{f_n(z)}^p \, dV_{\beta}(z)   \nonumber \\
   & = C \left( \int_{\U \backslash \U_{\delta}} \widehat{\mu}(z)  \n{f_n(z)}^p \, dV_{\beta}(z) 
       +  \int_{ \U_{\delta}} \widehat{\mu}(z) \n{f_n(z)}^p \, dV_{\beta}(z) \right)  \label{term}
\end{align}
by Lemma \ref{CarlesonLemma}. Since $\widehat{\mu}(z)$ is a continuous function on $\U \backslash \U_{\delta}$ 
and $\U \backslash \U_{\delta}$ is a compact set, there exists a constant $M_{\delta}>0$ such that 
\begin{align*}
\sup_{z \in \U \backslash \U_{\delta}} \widehat{\mu}(z) \leq M_{\delta} .
\end{align*}
Therefore, the first term of (\ref{term}) is less than or equal to $C M_{\delta} \varepsilon$ if $n \geq N$. 
On the other hand, since $\{ f_n \}$ is a bounded sequence in $L^p_a(\mathcal{U}, dV_{\beta})$, 
there exists a constant $M>0$ such that 
\begin{align*}
\int_{\U}  \n{f_n(z)}^p \, dV_{\beta}(z)  \leq M 
\end{align*}
for all $n \in \mathbb{N}$. Therefore, the second term of (\ref{term}) is less than or equal to $C M \varepsilon$. 
Hence, we obtain 
\begin{align*}
\int_{\U} \n{f_n(z)}^p \, d\mu(z) \leq C (M+M_{\delta}) \varepsilon
\end{align*}
for any $n \geq N$. Therefore, we obtain 
\begin{align*}
\lim_{n \rightarrow \infty} \int_{\U} \n{f_n(z)}^p \, d\mu(z) =0 .
\end{align*}
Hence,  $(iii) \Longrightarrow (i)$ holds.
\end{proof}

We can show that these conditions are also equivalent to the following condition (cf. \cite[Theorem 5.1]{Yamaji}): 
(iv) The Toeplitz operator $T_{\mu}$ is compact on $L^2_a(\mathcal{U},dV_{\beta})$.


\section{Relation between Carleson measures and composition operators} \label{sect40}
\subsection{Criterion of boundedness} \label{sect41} 
From (\ref{DefC2}), we see that $C_{\varphi}$ is a bounded operator on $L^p_a(\mathcal{U},dV_{\beta})$ if and only if 
the pull-back measure $\mu_{\varphi,\beta}$ is a Carleson measure for $L^p_a(\mathcal{U},dV_{\beta})$. 
By Theorem \ref{W-Carlprop}, the property of being a Carleson measure is independent of $p$.
Hence, the boundedness of $C_{\varphi}$ on $L^p_a(\U,dV_{\beta})$ is also independent of $p$.
We summarize the characterization of the boundedness of $C_{\varphi}$ on $L^p_{a}(\U,dV_{\beta})$ as follows. 
\begin{Lem} \label{CompBddness} 
Let $\beta > \bmin$. Then, the following conditions are all equivalent. \\
$(i)$ \ $C_{\varphi}$ is a bounded operator on $L^p_{a}(\U,dV_{\beta})$. \\
$(ii)$ \ The pull-back measure $\mu_{\varphi,\beta}$ is a Carleson measure for $L^p_a(\mathcal{U},dV_{\beta})$. \\
$(iii)$ \ $\widetilde{\mu_{\varphi,\beta}}$ is a bounded function on $\U$. \\
$(iv)$ \ $\widehat{\mu_{\varphi,\beta}}$ is a bounded function on $\U$. \\
$(v)$ \ The function 
\begin{align*}
 F_{\ru,\beta}(z)  := \int_{B(z,\ru)} \n{k_z^{(\beta)} (w)}^2 \, d\mu_{\varphi,\beta} (w)
\end{align*}
is bounded on $\mathcal{U}$. 
\end{Lem}

\begin{proof}
The equivaliance of $(i)$ -- $(iv)$ follows from Theorem \ref{W-Carlprop}. 
Moreover, $(iii) \Longrightarrow (v)$ is trivial and $(v) \Longrightarrow (iv)$ follows from (\ref{W-Carlesonequiv2.3}). 
\end{proof}

If $C_{\varphi}$ is bounded, we have the following estimate. 
\begin{Lem} \label{Kvarphi} 
Assume that $C_{\varphi}$ is bounded on $L^p_a(\mathcal{U},dV_{\beta})$ for some $p>0$ and $\beta >\beta_0$. 
Then there exists a positive constant $C$ such that 
$$ K_{\mathcal{U}} \left( \varphi(z), \varphi(z) \right)   \leq  C {K_{\mathcal{U}}(z,z)}  $$
for any $z \in \mathcal{U}$. 
\end{Lem}
\begin{proof}
By Lemma \ref{CompBddness}, it is enough to consider $p=2$. 
By (\ref{Adj}), we have 
\begin{align*}
C_{\varphi}^{\ast} k_z^{(\beta)}(w) 
  &= \inner{k_z^{(\beta)}}{C_{\varphi}K_{w}^{(\beta)}}_{L^2(dV_{\beta})}  
  = K_{\U}^{(\beta)}(z,z)^{-\frac{1}{2}} \, \overline{\langle C_{\varphi}K_{w}^{(\beta)} , K_z^{(\beta)}\rangle}_{L^2(dV_{\beta})}  \\
  &= K_{\U}^{(\beta)}(z,z)^{-\frac{1}{2}} \, \overline{C_{\varphi}K_{w}^{(\beta)}(z)}  
  = \frac{K_{\U}^{(\beta)}(w,\varphi(z))}{K_{\U}^{(\beta)}(z,z)^{\frac{1}{2}}}  
\end{align*}
Therefore, we have 
\begin{align}
 \N{C_{\varphi}^{\ast} k_{z}^{(\beta)}}^2_{L^2(dV_{\beta})} 
    =  \frac{K_{\U}^{(\beta)} \left( \varphi(z), \varphi(z) \right) }{K_{\U}^{(\beta)}(z,z)} 
    =  \left( \frac{K_{\mathcal{U}}(\varphi(z), \varphi(z))}{K_{\mathcal{U}}(z,z)} \right)^{1+\beta} . \label{KBDD}
\end{align}
Since $C_{\varphi}$ is a bounded operator on $L^2_a(\mathcal{U},dV_{\beta})$ and $\Vert k_{z}^{(\beta)} \Vert_{L^2(dV_{\beta})} =1$, 
the left hand side of (\ref{KBDD}) is less than or equal to a positive constant $C$. 
\end{proof}

\begin{Thm} \label{BdCp}
If $C_{\varphi}$ is a bounded operator on $L^q_{a}(\U,dV_{\sita})$ for some $q>0$ and $\sita> \bmin$, 
then $C_{\varphi}$ is a bounded operator on $L^p_{a}(\U,dV_{\naka})$ for any $p>0$ and $\naka \geq \sita$. 
\end{Thm}

\begin{proof}
The boundedness of $C_{\varphi}$ on $L^q_{a}(\U,dV_{\sita})$ (resp. $L^p_{a}(\U,dV_{\ue})$) is equivalent to 
the boundedness of $\Fbel{\sita}$ and $\Fru{\sita}$ (resp. $\Fbel{\ue}$ and $\Fru{\ue}$) by Lemma $\ref{CompBddness}$. 
Therefore, it is sufficient to prove 
\begin{eqnarray}
\Fbel{\sita}(z) \geq C \Fru{\ue}(z) . \label{bdd.goal}
\end{eqnarray}
Since
\begin{eqnarray*}
K_{\mathcal{U}} (\varphi(w),\varphi(w)) \leq C K_{\mathcal{U}} (w,w)  
\end{eqnarray*}
by Lemma $\ref{Kvarphi}$, we have 
\begin{align*}
dV_{\sita}(w) 
     &= K_{\mathcal{U}} (w,w)^{\ue - \sita} \, dV_{\ue}(w) \\
    &\geq C K_{\mathcal{U}} (\varphi(w),\varphi(w))^{\ue - \sita} \, dV_{\ue}(w) .
\end{align*}
Hence, we obtain 
\begin{align}
{} & \Fbel{\sita}(z)   \nonumber \\
  &=     K_{\mathcal{U}} (z,z)^{-(1+ \sita)} \int_{\U} \n{K_{\U} (z,\varphi(w))}^{2(1+\sita)} \, dV_{\sita}(w)  \nonumber  \\
  &\geq  C K_{\mathcal{U}} (z,z)^{-(1+ \sita)} \int_{\U} \n{K_{\U} (z,\varphi(w))}^{2(1+\sita)}
             K_{\mathcal{U}} (\varphi(w),\varphi(w))^{\ue - \sita} \, dV_{\ue}(w) . \label{Bdd.1}
\end{align}
By the definition of the pull-back measure, the right hand side of (\ref{Bdd.1}) is equal to 
\begin{align}
 {} &     C K_{\mathcal{U}} (z,z)^{-(1+ \ue)} \int_{\U} \n{K_{\U} (z,w)}^{2(1+\ue)} 
   \left( \frac{K_{\U}(w,w) K_{\U}(z,z) }{\n{K_{\U}(z,w)}^2} \right)^{\ue - \sita} \, dV_{\ue}(w)   \nonumber  \\
  &\geq  C K_{\mathcal{U}} (z,z)^{-(1+ \ue)} \int_{B(z,\ru)} \n{K_{\U} (z,w)}^{2(1+\ue)} 
   \left( \frac{K_{\U}(w,w) K_{\U}(z,z) }{\n{K_{\U}(z,w)}^2} \right)^{\ue - \sita} \, dV_{\ue}(w)  .  \label{Bdd.2}
\end{align}
Since $w \in B(z,\ru)$, we have 
\begin{align*}
\frac{K_{\U}(w,w) K_{\U}(z,z) }{\n{K_{\U}(z,w)}^2}
  = \n{\frac{K_{\U}(w,w)}{K_{\U}(z,w)} \frac{K_{\U}(z,z) }{K_{\U}(z,w)}} \geq C_{\ru}^{-2} 
\end{align*}
by Theorem \ref{Ishi}. Therefore, (\ref{Bdd.2}) is greater than or equal to 
\begin{align*}
   C K_{\mathcal{U}} (z,z)^{-(1+ \ue)} \int_{B(z,\ru)} \n{K_{\U} (z,w)}^{2(1+\ue)} \, dV_{\ue}(w) = C \Fru{\ue}(z) .
\end{align*}
Hence, (\ref{bdd.goal}) holds. 
\end{proof}


\subsection{Criterion of compactness} \label{sect41-5} 
Let $p>0$. We say that $C_{\varphi}$ is compact on $L^p_a(\mathcal{U},dV_{\beta})$ if 
the image under $C_{\varphi}$ of any subset of $L^p_a(\mathcal{U},dV_{\beta})$ is a relatively compact subset. 
We see that $C_{\varphi}$ is compact on $L^p_a(\mathcal{U},dV_{\beta})$ if and only if
\begin{align}
\lim_{k \rightarrow \infty} \int_{\U} \n{C_{\varphi}f_k(w)}^p \, dV_{\beta}(w) =0 \label{defcompcpt}
\end{align}
holds whenever $\{ f_k \}$ is a bounded sequence in $L^p_a(\mathcal{U}, dV_{\beta})$ that converges to $0$ 
uniformly on each compact subset of $\U$ 
(for the case that $\U =\mathbb{D}$, see \cite[Proposition 3.1]{C-M}). 
Since (\ref{defcompcpt}) is equivalent to 
\begin{align*}
\lim_{k \rightarrow \infty} \int_{\U} \n{f_k(w)}^p \, d\mu_{\varphi,\beta}(w) =0 ,
\end{align*}
$C_{\varphi}$ is a compact operator on $L^p_a(\mathcal{U},dV_{\beta})$ if and only if 
$\mu_{\varphi,\beta}$ is a vanishing Carleson measure for $L^p_a(\mathcal{U}, dV_{\beta})$. 
By Theorem \ref{W-VanCarl}, the property of being a vanishing Carleson measure is independent of $p$. 
Hence, the compactness of $C_{\varphi}$ on $L^p_a(\U,dV_{\beta})$ is also independent of $p$. 
We note the characterlization of the compactness of $C_{\varphi}$ on $L^p_{a}(\U,dV_{\beta})$. 
%
\begin{Lem} \label{Compcptness} 
Let $\beta > \bmin$. Then, the following conditions are all equivalent. \\
$(i)$ \ $C_{\varphi}$ is a compact operator on $L^p_{a}(\U, dV_{\beta})$. \\
$(ii)$ \ $\mu_{\varphi,\beta}$ is a vanishing Carleson measure for $L^p_a(\U, dV_{\beta})$. \\
$(iii)$ \ $\displaystyle{\lim_{z \rightarrow \partial \U} \widetilde{\mu_{\varphi,\beta}}(z) = 0}$.\\
$(iv)$ \ $\displaystyle{\lim_{z \rightarrow \partial \U} \widehat{\mu_{\varphi,\beta}}(z) = 0}$.\\
$(v)$ \  $\displaystyle{\lim_{z \rightarrow \partial \U} F_{\ru,\beta}(z) = 0 }$. 
\end{Lem}


\begin{Thm} \label{CptCp}
If $C_{\varphi}$ is a compact operator on $L^q_{a}(\U, dV_{\sita})$ for some $q>0$ and $\sita> \bmin$, 
then $C_{\varphi}$ is a compact operator on $L^p_{a}(\U, dV_{\naka})$ for any $p>0$ and $\ue \geq \sita$. 
\end{Thm}
\begin{proof}
By Lemma \ref{Compcptness}, it is enough to prove 
\begin{align*}
\lim_{z \rightarrow \partial \U} \widetilde{\mu_{\varphi,\sita}}(z) = 0
 \Longrightarrow \lim_{z \rightarrow \partial \U} F_{\ru,\ue}(z) = 0 .
\end{align*}
This follows from (\ref{bdd.goal}). 
\end{proof}

\section{Some equalities} \label{sect4.00}
\subsection{Equality for a homogeneous Siegel domain} \label{sect22}
In order to characterlize the compactness of the composition operators on $L^p_a(\U,dV_{\beta})$, 
we use an integral formula on a homogeneous Siegel domain. 
First, we recall notation and properties of the homogeneous Siegel domains following \cite{B-K} and \cite{Gin}. 
Let $\Omega \subset \R^n$ be a convex cone not containing any straight lines and 
$F: \C^m \times \C^m \longrightarrow \C^n$ a Hermitian form such that 
$F(u,u) \in {\rm Cl}(\Omega) \backslash \{ 0 \} $, where ${\rm Cl}(\Omega)$ is the closure of $\Omega$. 
Then, the Siegel domain $\mathcal{D}$ is defined by 
\begin{align*}
\Siegel = \left\{ (\xi,\eta) \in \C^n \times \C^m \, \bigg| \, \frac{\xi- \overline{\xi}}{2i} -F(\eta,\eta) \in \Omega \right\} .
\end{align*}
It is known that every bounded homogeneous domain is holomorphically equivalent to a homogeneous Siegel domain \cite{VGS}. 

Let $l$ be the rank of $\Omega$. 
For $1 \leq j \leq l$, let $n_j \geq 0, q_j\geq 0$ and $d_j \leq0$ be real numbers defined in \cite{Gin} 
(These notations are also used in \cite{B-K}. Note that $d_j$ in \cite{I-Y} is $-d_j$ in the present notation). 
We write $\UDL{n}$ by the vector of $\R^l$ whose components are $n_j$. 
The notations $\UDL{q}$ and $\UDL{d}$ are used similarly. 
By using compound power functions defined in \cite[(2.3)]{Gin}, it is known that the Bergman kernel of $\Siegel$ is given by
\begin{eqnarray*}
K_{\Siegel}(\zeta,\zeta^{\prime})  = C \, \left( \frac{\xi-\overline{\xi^{\prime}}}{2i} - F(\eta,\eta^{\prime}) \right)^{2\UDL{d} -\UDL{q}} 
\quad (\zeta = (\xi,\eta), \, \zeta^{\prime} = (\xi^{\prime},\eta^{\prime}) )
\end{eqnarray*}
(see also \cite[Proposition II.1]{B-K}). 
For
\begin{eqnarray}
\beta > \bmin := - \min \left\{ \frac{n_j+2}{2(-2d_j+q_j)} \hspace{1mm} \bigg| \hspace{1mm} 1 \leq j \leq l \right\} ,  \label{WeightCond.}
\end{eqnarray}
we consider the weighted Bergman space 
$$ L^p_a(\Siegel,K_{\Siegel}(\zeta,\zeta)^{-\beta} dV(\zeta)) := L^p(\Siegel,K_{\Siegel}(\zeta,\zeta)^{-\beta} dV(\zeta)) \cap \mathcal{O}(\Siegel).  $$
By \cite[Theorem II.2]{B-K}, we see that $L^2_a(\Siegel,K_{\Siegel}(\zeta,\zeta)^{-\beta} dV(\zeta)) \neq \{  0 \}$ 
if (\ref{WeightCond.}) holds. 
For 
\begin{align*}
 \bint := \max \left\{ \frac{n_j}{2(-2d_j+q_j)} \hspace{1mm} \bigg| \hspace{1mm} 1 \leq j \leq l \right\} , 
\end{align*}
B{\'e}koll{\'e} and Kagou showed the following integral formula. 
\begin{Lem}[\mbox{\cite[Corollary II.4]{B-K}}]  \label{Res.IYY}
Let $\beta>\bmin$ and $\alpha > \beta + \bint$. Then, one has 
\begin{eqnarray}
 \int_{\mathcal{D}} \n{K_{\mathcal{D}} (\zeta,\zeta^{\prime})}^{1+\alpha}  K_{\mathcal{D}} (\zeta^{\prime},\zeta^{\prime})^{-\beta} \, dV(\zeta^{\prime}) 
 = C_{\mathcal{D}}(\alpha,\beta) K_{\mathcal{D}} (\zeta,\zeta)^{\alpha-\beta}  , \label{IYYint}
\end{eqnarray}
where $C_{\mathcal{D}}(\alpha,\beta)$ is a positive function of $\alpha$ and $\beta$. 
\end{Lem}
We shall obtain an equality of a minimal bounded homogeneous domain from Lemma \ref{Res.IYY}. 

\subsection{Equality for a minimal bounded homogeneous domain} \label{sect50--}
Let  $\mathcal{D}$ be a Siegel domain biholomorphic to $\mathcal{U}$ and 
$\Phi$ a biholomorphic map from $\mathcal{U}$ onto $\mathcal{D}$. 
We have an isometry
\begin{align*}
L^2_a(\Siegel,K_{\Siegel}(\zeta,\zeta)^{-\beta} dV(\zeta)) \ni f \longmapsto \det J(\Phi, \cdot)^{1+\beta} f \circ \Phi \in L^2_a(\U, dV_{\beta}) .
\end{align*}
In particular, $L^2_a(\U,dV_{\beta})\neq \{  0 \}$ for $\beta > \beta_{min}$.  

\begin{Lem}  \label{CulSiegel} 
Let $\beta> \bmin$ and $\alpha > \beta + \bint$. Then, one has
\begin{align*}
 \int_{\mathcal{U}} \n{K_{\mathcal{U}} (z,z^{\prime})}^{1+\alpha}   \n{\det J(\Phi,z^{\prime})}^{1+2\beta-\alpha}  \, dV_{\beta}(z^{\prime}) = C_{\mathcal{D}}(\alpha,\beta) K_{\mathcal{U}} (z,z)^{\alpha-\beta}  \n{\det J(\Phi,z)}^{1+2\beta-\alpha} 
\end{align*}
for any $z \in \mathcal{U}$.
\end{Lem}
\begin{proof}
Let $\zeta^{\prime} = \Phi(z^{\prime})$. Since 
\begin{align*}
dV_{\beta}(z^{\prime}) 
  &= K_{\mathcal{U}} (\Phi^{-1}(\zeta^{\prime}),\Phi^{-1}(\zeta^{\prime}))^{-\beta} \n{\det J(\Phi^{-1},\zeta^{\prime})}^2 dV(\zeta^{\prime})  \\
  &= \n{\det J(\Phi,\Phi^{-1}(\zeta^{\prime}))}^{-2(1+\beta)} K_{\mathcal{D}} (\zeta^{\prime},\zeta^{\prime})^{-\beta} dV(\zeta^{\prime}) ,
\end{align*}
we have 
\begin{align}
{} &   \int_{\mathcal{U}} \n{K_{\mathcal{U}} (z,z^{\prime})}^{1+\alpha} \n{\det J(\Phi,z^{\prime})}^{1+2\beta-\alpha} \, dV_{\beta}(z^{\prime}) \nonumber \\
   &=  \int_{\mathcal{D}} \n{K_{\mathcal{U}} (z,\Phi^{-1}(\zeta^{\prime}))}^{1+\alpha} 
    \n{\det J(\Phi,\Phi^{-1}(\zeta^{\prime}))}^{-(1+\alpha)}  K_{\mathcal{D}} (\zeta^{\prime},\zeta^{\prime})^{-\beta} \, dV(\zeta^{\prime}) . \label{intiq-1}
\end{align}
By transformation formula of the Bergman kernel, we have 
\begin{align*}
K_{\mathcal{U}} (z,\Phi^{-1}(\zeta^{\prime})) = K_{\mathcal{D}} (\Phi(z),\zeta^{\prime}) \det J(\Phi,z) \, \overline{\det J(\Phi,\Phi^{-1}(\zeta^{\prime}))} .
\end{align*}
Therefore, the right hand side of (\ref{intiq-1}) is equal to 
\begin{align}
  \int_{\mathcal{D}} \n{K_{\mathcal{D}} (\Phi(z),\zeta^{\prime})}^{1+\alpha}  \n{\det J(\Phi,z)}^{1+\alpha} 
     K_{\mathcal{D}} (\zeta^{\prime},\zeta^{\prime})^{-\beta} \, dV(\zeta^{\prime}) . \label{intiq-2}
\end{align}
By Lemma \ref{Res.IYY}, (\ref{intiq-2}) is equal to
\begin{align*}
 C_{\mathcal{D}}(\alpha,\beta)  \n{\det J(\Phi,z)}^{1+2\beta-\alpha} K_{\mathcal{U}} (z,z)^{\alpha-\beta} .
\end{align*}
\end{proof}

\begin{Cor} \label{Lp}
Let $\beta> \bmin$ and $\alpha > \beta + \bint$. For any $z \in \mathcal{U}$, the function
\begin{eqnarray*}
 g_z(w) := K_{\mathcal{U}} (w,z)^{\frac{1+\alpha}{2}}   \det J(\Phi,w)^{\frac{1+2\beta-\alpha}{2}}  
\end{eqnarray*}
is in $L^2_{a}(\mathcal{U},dV_{\beta})$. In particular, one has 
\begin{eqnarray*}
 \N{g_z}^2_{L^2(dV_{\beta})} = C_{\mathcal{D}}(\alpha,\beta) K_{\mathcal{U}} (z,z)^{\alpha-\beta}  \n{\det J(\Phi,z)}^{1+2\beta-\alpha} .
\end{eqnarray*}
\end{Cor}
\begin{proof}
We have 
\begin{eqnarray*}
 \N{g_z}_{L^2(dV_{\beta})}^2 
  =  \int_{\mathcal{U}}  \n{K_{\mathcal{U}} (z,w)}^{1+\alpha}   \n{\det J(\Phi,w)}^{1+2\beta-\alpha} \, K_{\mathcal{U}} (w,w)^{-\beta}  dV(w)  .
\end{eqnarray*}
By Lemma \ref{CulSiegel}, this is equal to 
$ C_{\mathcal{D}}(\alpha,\beta) K_{\mathcal{U}} (z,z)^{\alpha-\beta}  \n{\det J(\Phi,z)}^{1+2\beta-\alpha}$.
\end{proof}

By using Corollary \ref{Lp}, we construct a positive function that satisfies the condition of 
Schur's Theorem (see \cite[Theorem 3.6]{Zhu3}) in section \ref{sect-Lemma}. 


\section{Characterlization of the compactness of composition operators} \label{sect50}
\subsection{Proof of the characterization of the compactness} \label{sect50-}
By using the lemmas in section \ref{sect-Lemma}, we prove Theorem \ref{MainComp0}. 
\begin{Thm} \label{MainComp}
If $C_{\varphi}$ is bounded on $L^q_a(\mathcal{U},dV_{\sita})$, 
then the following conditions are equivalent for $\naka > \sita + \bint$. \\
$(i)$ \ $C_{\varphi}$ is a compact operator on $L^p_a(\mathcal{U},dV_{\naka})$. \\
$(ii)$ \ $\displaystyle{\lim_{ z \rightarrow \partial \mathcal{U}} \frac{K_{\mathcal{U}} \left( \varphi(z), \varphi(z) \right) }{K_{\mathcal{U}}(z,z)} = 0 }$. \label{CompositionCptness}
\end{Thm}

\begin{proof}
It is enough to prove $p=q=2$. 
First, we prove that $(i)$ implies $(ii)$. 
Assume that $C_{\varphi}$ is a compact operator on $L^2_a(\mathcal{U},dV_{\naka})$. Then, $C^{\ast}_{\varphi}$ is also compact. 
Since $\{ k_{z}^{(\naka)} \}$ converges to $0$ 
uniformly on compact subsets of $\U$ as $z \rightarrow \partial \mathcal{U}$, 
we have $\Vert C^{\ast}_{\varphi} k_z^{(\naka)} \Vert_{L^2(dV_{\beta})} \rightarrow 0$ as $z \rightarrow \partial \mathcal{U}$. 
From (\ref{KBDD}), we obtain $(ii)$. 

Next, we prove that $(ii)$ implies $(i)$. 
For $f \in L^2_a(\mathcal{U},dV_{\naka})$, let 
\begin{eqnarray*}
   S f(z) := \int_{\mathcal{U}} K_{\mathcal{U}}^{(\naka)} \left( \varphi(z), \varphi(w) \right)  f(w) \, dV_{\naka}(w) . 
\end{eqnarray*}
Since $C_{\varphi}$ is a bounded operator on $L^2_a(\mathcal{U},dV_{\naka})$, 
we have $C_{\varphi} C_{\varphi}^{\ast} = S$ by (\ref{CptKakikae}). 
Therefore the compactness of $C_{\varphi}$ is equivalent to the compactness of $S$. 
Hence, it is sufficient to prove that $S^+$ is a compact operator on $L^2(\mathcal{U},dV_{\naka})$, where
\begin{eqnarray*}
  S^+ f(z) := \int_{\mathcal{U}} \n{K_{\mathcal{U}}^{(\naka)} \left( \varphi(z), \varphi(w) \right)}  f(w) \, dV_{\naka}(w) 
\end{eqnarray*}
for $f \in L^2(\mathcal{U},dV_{\naka})$. 
For $r>0$, let $ \mathcal{U}_r := \{ z \in \mathcal{U} \mid {\rm dist} (z, \partial \mathcal{U}) < r \} $. We define 
\begin{align*}
 K^+_{1,r} (z,w) &:= \chi_{\mathcal{U} \backslash  \mathcal{U}_r} (w) \n{ K_{\mathcal{U}}^{(\naka)}  \left( \varphi(z), \varphi(w) \right)}  , \\
 K^+_{2,r} (z,w) &:= \chi_{\mathcal{U} \backslash  \mathcal{U}_r} (z) \, \chi_{\mathcal{U}_r} (w) \n{ K_{\mathcal{U}}^{(\naka)}  \left( \varphi(z), \varphi(w) \right)}   , \\
 K^+_{3,r} (z,w) &:= \chi_{\mathcal{U}_r} (z) \, \chi_{\mathcal{U}_r} (w) \n{ K_{\mathcal{U}}^{(\naka)}  \left( \varphi(z), \varphi(w) \right)}  ,
\end{align*}
and $S^+_{j,r}$ by integral operators on $L^2(\mathcal{U},dV_{\naka})$ with kernel $K^+_{j,r}$. 
Then, we have 
$$  S^+ = S^+_{1,r}+S^+_{2,r}+S^+_{3,r} . $$ 
We will prove that $S^+_{1,r}$ and $S^+_{2,r}$ are compact operators on $L^2(\mathcal{U},dV_{\naka})$ for any $r>0$ and 
$\N{S^+_{3,r}} \rightarrow 0$ as $r \rightarrow 0$ in section \ref{sect-Lemma}. 
Using these results, we see that $S^+$ is a compact operator on $L^2(\mathcal{U},dV_{\naka})$. 
\end{proof}

\subsection{Some Lemmas} \label{sect-Lemma}
In this subsection, we show some properties of the operators defined in the proof of Theorem \ref{MainComp}.
Since we assumed that $C_{\varphi}$ is bounded on $L^p_a(\mathcal{U},dV_{\sita})$, we have the following lemma. 
\begin{Lem}\ \label{Hilb-Schm-op}
The operators $S^+_{1,r}$ and $S^+_{2,r}$ are compact on $L^2(\mathcal{U},dV_{\naka})$. 
\end{Lem}
\begin{proof}
It is enough to prove $K^+_{1,r}$ and $K^+_{2,r}$ are in $L^2(\mathcal{U} \times \mathcal{U},dV_{\naka} \times dV_{\naka})$ 
(for example, see \cite[Theorem 3.5]{Zhu3}). 
For $w \in \U$, let 
$$ K_{\varphi(w)}^{(\naka)}(z)  :=  K_{\mathcal{U}}^{(\naka)} \left( z,\varphi(w) \right).  $$
Then, $K_{\varphi(w)}^{(\naka)} \in L^2_a(\mathcal{U}, dV_{\naka})$ and we have 
\begin{align}
\N{K^+_{1,r}}_{L^2(\mathcal{U} \times \mathcal{U})}^2 
     &=  \int_{\mathcal{U} \backslash  \mathcal{U}_r} \left\{ \int_{\mathcal{U}} \n{ K_{\mathcal{U}}^{(\naka)} \left( \varphi(z), \varphi(w) \right)}^2 \, dV_{\naka}(z) \right\}  \, dV_{\naka}(w) \nonumber  \\
     &=  \int_{\mathcal{U} \backslash  \mathcal{U}_r} \N{C_{\varphi} K_{\varphi(w)}^{(\naka)}}^2_{L^2( dV_{\naka})} dV_{\naka}(w) .  \label{Cpt1}
\end{align} 
Since $C_{\varphi}$ is a bounded operator on $L^p_a(\mathcal{U}, dV_{\sita})$, 
$C_{\varphi} $ is bounded on $L^2_a(\mathcal{U}, dV_{\naka})$ by Theorem $\ref{BdCp}$. 
Hence, we have 
\begin{align}     
 \N{C_{\varphi} K_{\varphi(w)}^{(\naka)}}^2_{L^2( dV_{\naka})}
    \leq C  \N{ K_{\varphi(w)}^{(\naka)}}^2_{L^2( dV_{\naka})} 
     =   C K_{\mathcal{U}}^{(\naka)} \left( \varphi(w), \varphi(w) \right) 
     \leq   C K_{\mathcal{U}}(w,w)^{1+\naka}   ,  \label{Cpt2}
\end{align}
where the last inequality follows from Lemma $\ref{Kvarphi}$. Substituting $(\ref{Cpt2})$ to $(\ref{Cpt1})$, we obtain 
\begin{align*}
\N{K^+_{1,r}}_{L^2(\mathcal{U} \times \mathcal{U})}^2 
  &\leq  C \int_{\mathcal{U} \backslash  \mathcal{U}_r}  K_{\mathcal{U}} (w,w)^{1+\naka}  \, dV_{\naka}(w)   \\
     &=  C \int_{\mathcal{U} \backslash  \mathcal{U}_r}  K_{\mathcal{U}} (w,w)  \, dV(w)  \\
      &< \infty . 
\end{align*}
Similarly, we have $\N{K^+_{2,r}}_{L^2(\mathcal{U} \times \mathcal{U})} < \infty$. 
\end{proof}

Assume that $\naka > \sita + \bint$ and $C_{\varphi}$ is a bounded operator on $L^p_a(\mathcal{U}, dV_{\sita})$. 
Then, we obtain Lemma \ref{CompMainLemma}, 
which plays an impotant role in the proof of Theorem \ref{MainComp}. 
The assumption $\naka > \sita + \bint$ is only used to prove Lemma \ref{CompMainLemma}. 
\begin{Lem} \label{CompMainLemma}
For $z \in \mathcal{U}$, let 
\begin{eqnarray*}
h(z) := K_{\mathcal{U}}(z,z)^{\naka-\sita} \n{\det J (\Phi, \varphi(z))}^{1+2\sita-\naka} .
\end{eqnarray*}
Then, one has
\begin{align}
\int_{\mathcal{U}} K^+_{3,r}(z,w) h(w) \, dV_{\naka}(w)
   \leq C \, \chi_{\mathcal{U}_r} (z) \left( \frac{K_{\mathcal{U}} \left( \varphi(z), \varphi(z) \right) }{K_{\mathcal{U}}(z,z)} \right)^{\naka-\sita}  h (z). \label{norngoal}
\end{align}
\end{Lem}
\begin{proof}
For $z \in \mathcal{U}$, we have 
\begin{align}
{} & \int_{\mathcal{U}} K^+_{3,r}(z,w) h(w) \, dV_{\naka}(w)  \nonumber \\
     &=  \int_{\mathcal{U}} \chi_{\mathcal{U}_r} (z) \, \chi_{\mathcal{U}_r} (w) \n{K_{\mathcal{U}} (\varphi(z),\varphi(w))}^{1+\naka}  \det J( \Phi,\varphi(w))^{1+2\sita-\naka} \, dV_{\sita}(w) . \label{Kai1}
\end{align}
Here, we define a holomorphic function $g_z$ by
$$  g_z(w) := \left\{ K_{\mathcal{U}} (w, \varphi(z))^{1+\naka} \det J( \Phi, w)^{1+2\sita-\naka} \right\}^{\frac{1}{2}} . $$
Then, the right hand side of (\ref{Kai1}) is equal to 
\begin{align}
  \chi_{\mathcal{U}_r} (z) \int_{\mathcal{U}_r} \n{g_z (\varphi(w))}^2  \, dV_{\sita}(w)  
    \leq   \chi_{\mathcal{U}_r} (z) \int_{\mathcal{U}} \n{C_{\varphi} g_z (w)}^2  \, dV_{\sita}(w) .  \label{Kai2}  
\end{align}
Since $\sita > \bmin$ and $\naka > \sita + \bint$, the function $g_z$ is in $L^2_a(\mathcal{U}, dV_{\sita})$ by Corollary \ref{Lp}. 
Moreover, 
since $C_{\varphi}$ is a bounded operator on $L^p_a(\mathcal{U}, dV_{\sita})$ by assumption, 
$C_{\varphi} $ is bounded on $L^2_a(\mathcal{U}, dV_{\sita})$ by Theorem $\ref{BdCp}$. 
Therefore, we have 
\begin{align}
\int_{\mathcal{U}} K^+_{3,r}(z,w) h(w) \, dV_{\naka}(w) 
   &\leq  \chi_{\mathcal{U}_r} (z)  \N{C_{\varphi} g_z }^2_{L^2(dV_{\sita})}  \nonumber \\
   &\leq C  \chi_{\mathcal{U}_r} (z) \N{ g_z }^2_{L^2(dV_{\sita})} . \label{norm1}
\end{align}
On the other hand, we have 
\begin{align}
\N{ g_z }^2_{L^2(dV_{\sita})} =  K_{\mathcal{U}} ( \varphi(z), \varphi(z))^{\naka-\sita} \det J( \Phi,\varphi(z))^{1+2\sita-\naka}  \label{norm2}
\end{align}
by Corollary \ref{Lp}. Substituting (\ref{norm2}) to (\ref{norm1}), we obtain (\ref{norngoal}). 
\end{proof}

As we have already noted, the following lemma completes the proof of Theorem \ref{MainComp}. 
\begin{Lem} \label{PropConv}
One has $\N{S^+_{3,r}} \rightarrow 0$ as $r \rightarrow 0$. 
\end{Lem}

\begin{proof}
Put 
\begin{eqnarray*}
M(r) := \sup_{z \in \mathcal{U}_r}  \left\{ \frac{K_{\mathcal{U}} \left( \varphi(z), \varphi(z) \right) }{K_{\mathcal{U}}(z,z)} \right\}^{\naka-\sita}   .   
\end{eqnarray*}
By Lemma \ref{CompMainLemma}, we have 
\begin{eqnarray*}
\int_{\mathcal{U}} K^+_{3,r}(z,w) h(w) \, dV_{\naka}(w)
   \leq C \, M(r) h (z) .
\end{eqnarray*}
By using Schur's Theorem, $S^+_{3,r}$ is a bounded operator on $L^2_a(\mathcal{U},dV_{\naka})$ with norm 
not exceeding $CM(r)$. 
By (ii) of Theorem \ref{MainComp}, 
we obtain $M(r) \rightarrow 0$ as $r \rightarrow 0$. Hence we have $\N{S^+_{3,r}} \rightarrow 0$ as $r \rightarrow 0$.
\end{proof}

\section{Examples}\label{sect60}
We apply Theorem \ref{MainComp0} for some examples. 
\subsection{The unit ball}
Let $\mathcal{U}$ be the unit ball $\mathbb{B}^d$. 
Then, we have $l=1$ and $n_1=0, d_1=-1,q_1=d-1$. 
Hence, we have $\bmin =- \frac{1}{d+1}$ and $\bint =0$. 
Therefore, we obtain the following; 
\begin{Cor}[\mbox{\cite[Theorem 4.1]{ZhuComp}}]  \label{BallCor} 
Suppose $\naka > - \frac{1}{d+1}$. 
If the composition operator $C_{\varphi}$ is bounded on $L^q_a(\mathbb{B}^d, dV_{\sita})$ for some $- \frac{1}{d+1} < \sita < \naka$, 
then $C_{\varphi}$ is compact on $L^p_a(\mathbb{B}^d,dV_{\naka})$ if and only if 
\begin{align*}
 \lim_{ z \rightarrow \partial \mathbb{B}^d} \frac{K_{\mathbb{B}^d} \left( \varphi(z), \varphi(z) \right) }{K_{\mathbb{B}^d}(z,z)} = 0 . 
\end{align*}
\end{Cor}
\subsection{The Harish-Chandra realization of irreducible bounded symmetric domain}
Let $\Omega$ be an irreducible bounded symmetric domain in its Harish-Chandra realization, 
$r$ the rank of $\Omega$ and $a,b$ nonnegative integers defined in \cite{LuHu}. 
Then, we see that 
$l=r$ and $n_j=a(r-j), d_j=-1-\frac{a(r-1)}{2}, q_j=b$ for $1 \leq j \leq r$. 
Hence, we obtain $\bmin =- \frac{1}{N}$ and $\bint = \frac{a(r-1)}{2N}$, where $N := q_j-2d_j = a(r-1)+b+2$ is the genus of $\Omega$. 

\begin{Cor}[\mbox{\cite[Theorem]{LuHu}}]  \label{Clas2Cor} 
Suppose $\sita > - \frac{1}{N}$. 
If the composition operator $C_{\varphi}$ is bounded on $L^q_a(\Omega, dV_{\sita})$ 
for some $q>0$ and $\sita + \frac{a(r-1)}{2N}< \naka$, 
then $C_{\varphi}$ is compact on $L^p_a(\Omega,dV_{\naka})$ if and only if 
\begin{align*}
 \lim_{ z \rightarrow \partial \Omega} \frac{K_{\Omega} \left( \varphi(z), \varphi(z) \right) }{K_{\Omega}(z,z)} = 0 . 
\end{align*}
\end{Cor}
\subsection{The polydisk}
Let $\mathcal{U}$ be the polydisk $\mathbb{D}^m := \mathbb{D} \times \cdots \times \mathbb{D}$. 
Then, we have $l=m$ and $n_j=0, d_j=-1, q_j=0$ for $1 \leq j \leq m$. 
Hence, we have $\bmin =- \frac{1}{2}$ and $\bint =0$. 
\begin{Cor}\label{PolydiskCor} 
Suppose $\naka >- \frac{1}{2}$. 
If the composition operator $C_{\varphi}$ is bounded on $L^q_a(\mathbb{D}^m, dV_{\sita})$ for some $- \frac{1}{2} < \sita < \naka$, 
then $C_{\varphi}$ is compact on $L^p_a(\mathbb{D}^m,dV_{\naka})$ if and only if 
\begin{align*}
 \lim_{ z \rightarrow \partial \mathbb{D}^m} \frac{K_{\mathbb{D}^m} \left( \varphi(z), \varphi(z) \right) }{K_{\mathbb{D}^m}(z,z)} = 0 . 
\end{align*}
\end{Cor}
\subsection{A nonsymmetric minimal homogeneous domain}
Let $T_{\Omega}$ be the tube domain over Vinberg cone. 
It is known that $T_{\Omega}$ is a nonsymmetric homogeneous domain (see \cite{B-N}). 
By \cite[Proposition 3.8]{I-K}, the representative domain $\mathcal{U}$ of $T_{\Omega}$ 
is a nonsymmetric minimal bounded homogeneous domain with a center $0$. 
In this case, we have $l=3$ and $\UDL{n} = (2,0,0), \UDL{d} = (-2,-\frac{3}{2},-\frac{3}{2}), \UDL{q} = (0,0,0)$. 
Hence, we have $\bmin =- \frac{1}{3}$ and $\bint =\frac{1}{4}$. 

\begin{Cor}\label{VinbergCor} 
Suppose $\sita >- \frac{1}{3} $ and the composition operator $C_{\varphi}$ is bounded on $L^q_a(\U, dV_{\sita})$. 
For $ \naka > \sita + \frac{1}{4} $, $C_{\varphi}$ is compact on $L^p_a(\U,dV_{\naka})$ if and only if 
\begin{align*}
 \lim_{ z \rightarrow \partial \U} \frac{K_{\U} \left( \varphi(z), \varphi(z) \right) }{K_{\U}(z,z)} = 0 . 
\end{align*}
\end{Cor}

\subsection*{Acknowledgment}
The author would like to express my gratitude to Professor H.~Ishi for advices and suggestions. 
The author would also like to thank to Professor T.~Ohsawa for helpful discussions. 


\end{document}